\documentclass[11pt]{amsart}
\usepackage{amsmath}
\usepackage{amssymb}
\usepackage{amsthm}
\usepackage{enumerate}
\usepackage[mathscr]{eucal}

\makeatletter \oddsidemargin.9375in \evensidemargin \oddsidemargin
\itshape

\newtheorem{theorem}{Theorem}[section]

\theoremstyle{definition}
\newtheorem{definition}[theorem]{Definition}
\newtheorem{example}[theorem]{Example}
\newtheorem{note}[theorem]{Note}

\theoremstyle{remark}
\newtheorem{remark}[theorem]{Remark}
\numberwithin{equation}{section}

\makeatletter
\@namedef{subjclassname@2020}{\textup{2020} Mathematics Subject Classification}
\makeatother

\allowdisplaybreaks
\begin{document}
\setcounter{page}{1}

\title[$\mathcal{I}$ and $\mathcal{I}^{*}$-soft convergence in soft topological spaces]{$\mathcal{I}$ and $\mathcal{I}^{*}$-soft convergence in soft Topological spaces}

\author[P. MALIK AND A. PAUL]{Prasanta Malik$^1$$^{*}$ and Anirban Paul$^2$$^{*}$}

\address{$^{1}$ Department of Mathematics, The University of Burdwan, Golapbag, Burdwan-713104,
West Bengal, India.}
\email{pmjupm@yahoo.co.in}

\address{$^{2}$ Department of Mathematics, The University of Burdwan, Golapbag, Burdwan-713104,
West Bengal, India.}
\email{anirban15997ap@gmail.com}


\keywords{Soft set, $\mathcal{I}$-soft convergence, $\mathcal{I}$-soft
 limit point, $\mathcal{I}$-soft cluster point \\
\indent $^{*}$ Corresponding author}
\maketitle

\begin{abstract}
In this paper, we introduce the notions of $\mathcal{I}$ and $\mathcal{I}^{*}$-soft convergence of sequences of soft points in
soft topological spaces and study some basic properties of these notions. Also we introduce the notions of 
$\mathcal{I}$-soft limit points and $\mathcal{I}$-soft cluster points of a sequence of soft points in a soft topological space and study their
interrelationship.       
\end{abstract}

\section{\textbf{Introduction and Background}}
The concept of soft set theory was first introduced by Molodtsov \cite{M1} in the year 1999 and basic algebraic properties of soft sets were studied by Maji et al.\cite{Maji1}. For more primary works in this line one can see (\cite{Acar},\cite{Akt},\cite{Ali} etc.). Because of immense importance the notion of soft topology was introduced by {\c{C}}a{\u{g}}man et al. \cite{Cag1} in 2011. 

On the other hand the notion of statistical convergence of sequences of real numbers was first introduced  by 
H.Fast \cite{Fast1} and also independently by I.J. Schoenberg \cite{Sch1}. This notion of statistical convergence was further extended to ideal convergence by Kostyrko et al. \cite{Kos1} using the concept of ideal of subsets of $\mathbb N$. Because of great importance the notion of statistical convergence for sequences in topological spaces was introduced by Maio et al. \cite{Maio1} and the notion of ideal convergence was developed in topological space by Lahiri et al. \cite{Lah1}.

Recently in soft topological spaces the notion of soft convergence has been introduced by Bayramov et al. \cite{Bay1} and the notion of weighted statistical soft convergence has been introduced by Bayrama et al. \cite{Bay2} and some basic properties of these notions have been studied. In this paper we introduce and study the notion of $\mathcal{I}$-soft convergence of sequences of soft points in soft topological spaces, which extends both the notion of soft convergence and soft statistical convergence. Further we introduce the notion of $\mathcal{I}^{*}$-soft convergence and study it relationship with $\mathcal{I}$-soft convergence. Also we introduce the notions of $\mathcal{I}$-soft limit point and $\mathcal{I}$
-soft cluster point of a sequence of soft points in a soft topological space and study their interconnection. 

\section{\textbf{Basic definitions and notations}} 

In this section we first recall some basic definitions and notations related to statistical and ideal convergence.
\begin{definition} \cite{Fast1} \cite{Sch1}
Let $B\subset \mathbb{N}$ and $B(n)=\{k\in B: k\leq n\}$. Then $B$ is said to have natural density $d(B)$,
if $d(B) = \displaystyle{\lim_{n\rightarrow \infty}}\frac{\left|B(n)\right|}{n}$, where 
$\left|B(n)\right|$ denotes cardinality of the set $B(n)$.
\end{definition}

\begin{definition} \cite{Maio1}
Let $(X,\tau)$ be a topological space. A sequence $\{\eta_{n}\}_{n\in \mathbb N}$ in
$(X,\tau)$ is said to be statistically convergent to a point $x\in X$ if for every neighbourhood $U$ of $x$,
\begin{align*}
d(\{n\in \mathbb N :\eta_{n} \notin U\}) = 0.
\end{align*}									
In this case, we write $st-\displaystyle{\lim_{n\rightarrow \infty}}\eta_{n}=x$.
\end{definition}

\begin{definition} \cite{Kur}
Let $X$ be a non empty set and $\mathcal{I}$ be a collection of subsets of $X$. Then, $\mathcal{I}$ is said 
to be an ideal in $X$ if,
\begin{itemize}
\item[$(i)$] $\emptyset\in \mathcal{I}$,
\item[$(ii)$] $A\in \mathcal{I}~\mbox{and}~ B\subset A \implies B\in \mathcal{I}$,
\item[$(iii)$] $A\in \mathcal{I}~\mbox{and}~B\in \mathcal{I} \implies A\cup B\in \mathcal{I}$.
\end{itemize}
\end{definition}

An ideal $\mathcal{I}$ of $X$ is said to be non-trivial if $\mathcal{I}\neq \{\emptyset\}$ and $X\notin \mathcal{I}$.

A non-trivial ideal $\mathcal{I}$ of $X$ is said to be admissible if $\{x\}\in \mathcal{I}$, for every $x\in X$.

\begin{definition} \cite{Kur}
Let $X$ be a non-empty set and $\mathcal{F}$ be a non empty collection of subsets of $X$. Then, $\mathcal{F}$ is said 
to be a filter on $X$ if,
\begin{itemize}
\item[$(i)$] $\emptyset \notin \mathcal{F}$,
\item[$(ii)$] $A\in \mathcal{F}~\mbox{and}~ B\supset A \implies B\in \mathcal{F}$,
\item[$(iii)$] $A\in \mathcal{F}~\mbox{and}~B\in \mathcal{F} \implies A\cap B\in \mathcal{F}$.
\end{itemize}
\end{definition}

\begin{definition} \cite{Kos1}
Let $\mathcal{I}$ be a non-trivial ideal of a non-empty set $X$. Then the family of sets $\mathcal{F}(\mathcal{I})
=~\{A\subset X : \exists~B\in \mathcal{I}~\mbox{such that}~A=X-B\}$ is a filter on $X$, which is called filter 
associated with the ideal $\mathcal{I}$.
\end{definition}

Through out this paper, we consider that $\mathcal{I}$ as an admissible ideal of $\mathbb N$ unless mentioned otherwise.

\begin{definition} \cite{Lah1}
Let $(X, \tau)$ be a topological space. A sequence $\{\eta_{n}\}_{n\in \mathbb N}$ in $(X, \tau)$ is said to be $\mathcal{I}$-convergent to a point $x\in X$, if for every neighbourhood $U$ of $x$, $\{n\in \mathbb N : \eta_{n}\notin U\}\in \mathcal{I}$. 
\end{definition} 

\begin{definition} \cite{Lah1}
Let $(X, \tau)$ be a topological space. A sequence $\{\eta_{n}\}_{n\in \mathbb N}$ in $(X, \tau)$ is said to $\mathcal{I^{*}}$-convergent 
to a point $x\in X$ if there exists a set $N=\{n_{1}< n_{2}<...<n_{k}<...\}\in \mathcal{F}(\mathcal{I})$ such that 
$\displaystyle{\lim_{k\rightarrow \infty}} \eta_{n_{k}}= x$. 
\end{definition}

\begin{definition} \cite{Lah1}
Let $(X, \tau)$ be a topological space and $\{\eta_{n}\}_{n\in \mathbb N}$ be a sequence in $X$.
\begin{itemize}
\item[$(a)$] $x\in X$ is called an $\mathcal{I}$-limit point of $\{\eta_{n}\}_{n\in \mathbb N}$ if there exists a set $N=\{n_{1}< n_{2}<...\}\subset \mathbb N$ 
such that $N\notin \mathcal I$ and $\displaystyle{\lim_{k\rightarrow \infty}} \eta_{n_{k}}= x$.
\item[$(b)$] $x\in X$ is called an $\mathcal{I}$-cluster point of $\{\eta_{n}\}_{n\in \mathbb N}$ if for every neighbourhood $V$ containing $x$, 
$\{n: \eta_{n}\in V\}\notin \mathcal{I}$.
\end{itemize} 
\end{definition}

Following \cite{Cag1},\cite{Maji1},\cite{M1} we now recall basic concepts of soft set theory.

\begin{definition} \cite{Cag1} \cite{Maji1} \cite{M1}
Let $X$ be an initial universe set, $S$ be a set of parameters and $M\subset S$.
 A soft set $G_{M}$ over $X$ is defined by the set of all ordered pairs $G_{M}=\{(s,g_{M}(s)) : s\in S\}$, where
 $g_{M}:S\rightarrow \mathcal{P}(X)$ is a mapping such that $g_{M}(s)=\emptyset,~\forall~s\in S-M$.
\end{definition}

Throughout this paper we take $X$ as an initial universe set, $S$ as a set of parameters, $M\subset S$, $\mathcal{P}(X)$
as power set of $X$, $G_{M}$ as a soft set over $X$ and $\mathcal{A}(X)$ as a set of all soft sets over $X$ unless otherwise mentioned.

\begin{example}
Let the initial universe set $X$ consist of five universities, say, $X=\{x_{1}, x_{2}, x_{3}, x_{4}, x_{5}\}$.
Let $S=\{s_{1}, s_{2}, s_{3}, s_{4}\}$ be a set of parameters, where $s_{1}$ stands for `high quality research', $s_{2}$ stands for 
`excellent faculties', $s_{3}$ stands for `intelligent students' and $s_{4}$ stands for `modern research lab'. Let $M=\{s_{1}, s_{2}, s_{3}\}\subset S$ and $g_{M}:S\rightarrow \mathcal{P}(X)$ be given by

$g_{M}(s_{1})=\{x_{1}\}$, $g_{M}(s_{2})=\{x_{1}, x_{3}, x_{5}\}$, $g_{M}(s_{3})=\{x_{1}, x_{4}, x_{5}\}$ and $g_{M}(s_{4})= \emptyset$.

Note that $g_{M}(s_{i})$ literally represents the set of all universities having the property $s_{i},~i=1,2,3,4$.

Then $G_{M}=\{(s_{1},\{x_{1}\}), (s_{2}, \{x_{1}, x_{3}, x_{5}\}), (s_{3}, \{x_{1}, x_{4}, x_{5}\})\}$
 is a soft set over $X$. 
\end{example}

\begin{definition} \cite{Cag1} \cite{Maji1}
If $G_{M}\in \mathcal{A}(X)$ and $g_{M}(s)= \emptyset,~\forall~s\in S$, then $G_{M}$ is called an empty soft set or null
soft set and is denoted by $G_{\emptyset}~\mbox{or}~\widetilde{\emptyset}$.
\end{definition}

\begin{definition} \cite{Cag1} \cite{Maji1}
If $G_{M}\in \mathcal{A}(X)$ and $g_{M}(s)= X,~\forall~s\in M$, then $g_{M}$ is called a $M$-universal soft set and
is denoted by $G_{\widetilde{M}}$.

If $M=S$, then $G_{\widetilde{S}}$ is called an universal soft set or absolute soft set denoted by $\widetilde{X}$.
\end{definition}

\begin{definition} \cite{Cag1} \cite{Maji1}
Let $G_{M}, H_{N}\in \mathcal{A}(X)$. Then $G_{M}$ is said to be a soft subset of $H_{N}$, if $g_{M}(s)\subset h_{N}(s)$,
$\forall~s\in S$. We write $G_{M}\widetilde{\subset} H_{N}$. In this case, $H_{N}$ is also called a soft super set of $G_{M}$.

Two soft sets $G_{M}, H_{N}\in \mathcal{A}(X)$ are called soft equal if $g_{M}(s)= h_{N}(s),\\~\forall~ s\in S$. We write
$G_{M}= H_{N}$.

Two soft sets $G_{M}, H_{N}\in \mathcal{A}(X)$ are called soft unequal if there exists $s^{'}\in S$ such 
that $g_{M}(s^{'})\neq h_{N}(s^{'})$. We write
$G_{M}\neq H_{N}$.
\end{definition}

\begin{definition} \cite{Cag1} \cite{Maji1}
Let $G_{M}, H_{N}\in \mathcal{A}(X)$. The soft union of $G_{M}$ and $H_{N}$ is denoted by $G_{M}\widetilde{\cup} H_{N}$ and is defined by
$G_{M}\widetilde{\cup} H_{N}= \{(s, g_{M}(s)\cup h_{N}(s)): s\in S\}$; the soft intersection of $G_{M}$ and $H_{N}$ is denoted 
by $G_{M}\widetilde{\cap} H_{N}$ and is defined by $G_{M}\widetilde{\cap} H_{N}= \{(s, g_{M}(s)\cap h_{N}(s)): s\in S\}$; the 
soft difference of $G_{M}$ and $H_{N}$ is denoted by $G_{M}\widetilde{-} H_{N}$ and is defined by $G_{M}\widetilde{-} H_{N}= \{(s, g_{M}(s)- h_{N}(s)): s\in S\}$.

The soft complement of $G_{M}$ is denoted by $G_{M}^{\widetilde{C}}$ and is defined by $G_{M}^{\widetilde{C}}= \{(s, X- g_{M}(s)): s\in S\}$.
\end{definition}

\begin{definition} \cite{Bay1} \cite{Sujoy1}
A soft set $G_{M}$ in $\mathcal{A}(X)$ is called a soft point if $\exists~ s\in M$ and $x\in X$ such that $g_{M}(s)= \{x\}$
 and $g_{M}(s^{'})= \emptyset, ~\forall~ s^{'}\in S-\{s\}$. Such a soft point is denoted by $x_{s}^{M}$.

A soft point $x_{s}^{M^{*}}$ is said to be in the soft set $H_{M}$ if $x\in h_{M}(s)$. We write
 $x^{M^{*}}_{s}\widetilde{\in} H_{M}$. A soft point $x_{s}^{M^{*}}$ is said to be not in the soft set $H_{M}$, if 
$x\notin h_{M}(s)$ and we write $x_{s}^{M^{*}}\widetilde{\notin} H_{M}$.

Two soft points $x^{M^{*}}_{s_{1}}$ and $y^{M^{**}}_{s_{2}}$ in a soft set $H_{M}$ are called distinct if $s_{1}\neq s_{2}$ or $x\neq y$ or both
 hold.
\end{definition}

\begin{remark} 
Any soft set is soft union of all soft points in it.
\end{remark}

\begin{definition} \cite{Cag1}
Let $G_{M}$ be a soft set over $X$. A collection of soft subsets of $G_{M}$ is called a soft topology on $G_{M}$,
 if the following conditions are satisfied,
\begin{itemize}
\item[$(i)$] $\widetilde{\emptyset}\in \widetilde{\tau}$,
\item[$(ii)$] $G_{M}\in \widetilde{\tau}$,
\item[$(iii)$]  soft union of arbitrary members of $\widetilde{\tau}$ is again a member of $\widetilde{\tau}$,
\item[$(iv)$] soft intersection of finitely members of $\widetilde{\tau}$ is again a member of $\widetilde{\tau}$.
\end{itemize}
The pair $(G_{M},\widetilde{\tau})$ is called a soft topological space. When their is no confusion about the topology
 on $G_{M}$, we denote the soft topological space $(G_{M}, \widetilde{\tau})$ by $G_{M}$ only. Every member of $\widetilde{\tau}$
 is called a soft open set. A soft subset $H_{M^{'}}$ of $G_{M}$ is called soft closed set in the soft topological space $(G_{M},\widetilde{\tau})$,
 if $G_{M}\widetilde{-} H_{M^{'}}$ is soft open in the soft topological space $(G_{M}, \widetilde{\tau})$.
\end{definition}

\begin{definition} \cite{Cag1}
Let $(G_{M}, \widetilde{\tau})$ be a soft topological space and $H_{M^{'}}\widetilde{\subset} G_{M}$. Let ${\widetilde{\tau}}_{H_{M^{'}}}= \{R_{M^{''}}\widetilde{\cap} H_{M^{'}} : R_{M^{''}}\in \widetilde{\tau}\}$. 
Then $(H_{M^{'}}, {\widetilde{\tau}}_{H_{M^{'}}})$ is called a soft subspace of $(G_{M}, \widetilde{\tau})$.
\end{definition}

\begin{definition} \cite{Bay1} \cite{Hs1}
Let $(G_{M},\widetilde{\tau})$ be a soft topological space. A soft subset $K_{M^{'}}$ of a $G_{M}$ is said to be a soft neighbourhood of a soft point $x^{M^{*}}_{s}\widetilde{\in} G_{M}$ if there exists $ H_{M^{''}}\in \widetilde{\tau}$ such that $x^{M^{*}}_{s}~\widetilde{\in} H_{M^{''}}\widetilde{\subset} K_{M^{'}}$.
\end{definition}

\begin{definition} \cite{Cag1}
Let $(G_{M},\widetilde{\tau})$ be a soft topological space and $H_{M^{'}}$ be a soft subset of $G_{M}$. Then
 soft closure of $H_{M^{'}}$ is the soft intersection of all soft closed sets in $G_{M}$, containing the soft
 set $H_{M^{'}}$ and it is denoted by $cl(H_{M^{'}})$.
\end{definition}

\begin{theorem} \cite{Cag1}
A soft subset $H_{M^{'}}$ of $G_{M}$ is soft closed in $(G_{M},\widetilde{\tau})$ if and
 only if $cl(H_{M^{'}})= H_{M^{'}}$.
\end{theorem}

\begin{definition}\cite{Bay1}
Let $(G_{M}, \widetilde{\tau})$ be a soft topological space. A soft subset $H_{M^{'}}$ of $G_{M}$ is called soft dense
 in $(G_{M}, \widetilde{\tau})$ if $cl(H_{M^{'}})= G_{M}$.
\end{definition}

\begin{definition}\cite{Bay1}
A soft topological space $(G_{M}, \widetilde{\tau})$ is said to be soft separable if it has a countable 
soft subset, soft dense in $(G_{M}, \widetilde{\tau})$.  
\end{definition}

\begin{definition}\cite{Huss11}
Let $(G_{M}, \widetilde{\tau})$ be a soft topological space. If for any two distinct soft points 
$x^{M_{1}}_{s}$ and $y^{M_{2}}_{s^{'}}$ in $G_{M}$, there exist soft open sets $P_{M^{'}}$ and $Q_{M^{''}}$ such that
\begin{align*}
x^{M_{1}}_{s}\widetilde{\in} P_{M^{'}}, y^{M_{2}}_{s^{'}}\widetilde{\in} Q_{M^{''}}, x^{M_{1}}_{s}\widetilde{\notin} Q_{M^{''}}, y^{M_{2}}_{s^{'}}\widetilde{\notin} P_{M^{'}},
\end{align*} 
then $(G_{M},\widetilde{\tau})$ is called soft $T_{1}$ space.
\end{definition}

\begin{definition} \cite{Hs1}
Let $(G_{M},\widetilde{\tau})$ be a soft topological space. If for any two distinct soft points 
$x^{M_{1}}_{s}$ and $y^{M_{2}}_{s^{'}}$ in $G_{M}$, there exist soft open sets $P_{M^{'}}$ and $Q_{M^{''}}$ such that
\begin{align*}
x^{M_{1}}_{s}\widetilde{\in} P_{M^{'}}, y^{M_{2}}_{s^{'}}\widetilde{\in} Q_{M^{''}}, P_{M^{'}}\widetilde{\cap} Q_{M^{''}}=~\widetilde{\emptyset},
\end{align*}
then $(G_{M},\widetilde{\tau})$ is called soft Hausdorff space.
\end{definition}

\begin{definition} \cite{Bay1}
Let $(G_{M}, \widetilde{\tau})$ be a soft topological space. A sub-collection $\beta$ of $\widetilde{\tau}$ is
said to be soft base for $\widetilde{\tau}$, if every element in $\widetilde{\tau}$ can be written as soft 
union of some members of $\beta$.
\end{definition}

\begin{definition} \cite{Bay1}
Let $(G_{M}, \widetilde{\tau})$ be a soft topological space and $x_{s}^{M^{*}}$ be a soft point of $G_{M}$. Let $N_{x^{M^{*}}_{s}}$ be
 the set of all soft neighbourhood of $x^{M^{*}}_{s}$. Then a sub-collection $B_{x^{M^{*}}_{s}}$ of $N_{x_{s}^{M^{*}}}$ is said to be soft 
local base at $x^{M^{*}}_{s}$ if for any $H_{M^{'}}\in N_{x^{M^{*}}_{s}},~\exists~K_{M^{''}}\in B_{x^{M^{*}}_{s}}$ such that
 $x^{M^{*}}_{s}\widetilde{\in} K_{M^{''}}\widetilde{\subset} H_{M^{'}}$.  
\end{definition}

\begin{definition} \cite{Bay1}
A soft topological space $(G_{M},\widetilde{\tau})$ is said to be soft first countable at $x^{M^{0}}_{s}\widetilde{\in} G_{M}$ if 
there is a countable soft local base at $x^{M^{0}}_{s}$. 
\end{definition}

\begin{definition} \cite{Bay1}
A soft topological space $(G_{M},\widetilde{\tau})$ is said to be soft first countable if there is a countable soft local base at
 every soft point of $G_{M}$. 
\end{definition}

\begin{theorem} \cite{Bay1} \label{Th1}
Let $(G_{M}, \widetilde{\tau})$ be a soft topological space which is soft first countable. Then, for any soft point $x_{s}^{M^{*}}\widetilde{\in}$
 $G_{M}$, there exists a countable collection of soft local base $\{H_{M^{i}}\}_{i\in \mathbb N}$ at $x_{s}^{M^{*}}$ such that $H_{M_{1}}\widetilde{\supset} H_{M_{2}}\widetilde{\supset} \cdots \widetilde{\supset} H_{M_{i}}\widetilde{\supset}
 \cdots$.
\end{theorem}


\begin{definition} \cite{Bay1}
A soft topological space $(G_{M},\widetilde{\tau})$ is said to be soft second countable if $\widetilde{\tau}$ has a countable soft base.
\end{definition}

\begin{theorem}\cite{Bay1}
Let $(G_{M}, \widetilde{\tau})$ be a soft second countable. Then it is soft separable.
\end{theorem}

\begin{definition} \cite{Bay1} \cite{Hs1}
Let $(G_{M}, \widetilde{\tau})$ be a soft topological space . A sequence $\{\eta_{n}\}_{n\in \mathbb N}$ 
of soft points in $G_{M}$ is said to be soft convergent to a soft point $x^{M^{*}}_{s} \widetilde{\in} G_{M}$, if for any soft neighbourhood $H_{M^{'}}$ of
$x^{M^{*}}_{s} $, there exists $k\in \mathbb N$ such that $\eta_{n}\widetilde{\in} H_{M^{'}}, \forall n\geq k$. In this case, we write
 $\eta_{n}\underset{}{\rightarrow} x_{s}^{M^{*}}$ or $\displaystyle{\lim_{n\rightarrow \infty}}{\eta_{n}}=
 x_{s}^{M^{*}}$ and $x_{s}^{M^{*}}$ is called soft limit of $\{\eta_{n}\}_{n\in \mathbb N}$.
\end{definition}

\begin{definition}\cite{Bay2}
Let $(G_{M}, \widetilde{\tau})$ be a soft topological space . A sequence $\{\eta_{n}\}_{n\in \mathbb N}$ 
of soft points in $G_{M}$ is said to be statistically soft convergent to a soft point $x^{M^{*}}_{s}\widetilde{\in} G_{M}$, if for any soft neighbourhood
 $H_{M^{'}}$ of $x^{M^{*}}_{s}$, $d(\{n\in \mathbb N: \eta_{n}\widetilde{\notin} H_{M^{'}}\})=0$. In this case, we write
 st-$\displaystyle{\lim_{n\rightarrow \infty}}{\eta_{n}}= x_{s}^{M^{*}}$ or $\eta_{n} \underset{}{\stackrel{st}{\rightarrow}}
 x_{s}^{M^{*}}$. $x_{s}^{M^{*}}$ is called a statistical soft limit of $\{\eta_{n}\}_{n\in \mathbb N}$. 
\end{definition}

\begin{definition} \cite{Ar1}
Let $(G_{M},\widetilde{\tau})$ be a soft topological space. A soft point $x^{M^{*}}_{s}$ in $G_{M}$ is said to
 be a soft limit point of a soft subset $K_{M^{'}}$ of $G_{M}$, if for each soft
 neighbourhood $H_{M^{''}}$ of $x^{M^{*}}_{s}$, $(H_{M^{''}}\widetilde{-} \{x^{M^{*}}_{s}\})\widetilde{\cap} K_{M^{'}}~\neq~\widetilde{\emptyset}$.
\end{definition}

\begin{definition} \cite{Ar1}
Let $(G_{M},\widetilde{\tau})$ be a soft topological space. A soft point $x^{M^{*}}_{s}$ in $G_{M}$ is said to
 be a soft limit point of a sequence $w=\{\eta_{n}\}_{n\in \mathbb N}$ of soft points in $G_{M}$, if for each soft
 neighbourhood $H_{M^{'}}$ of $x^{M^{*}}_{s}$, the set $\{n\in \mathbb N: \eta_{n}\widetilde{\in} H_{M^{'}}\}$ is
 infinite.

The soft set of all soft limit points of $w=\{\eta_{n}\}_{n\in \mathbb N}$ is denoted by $L_{w}$.
\end{definition}


\section{\textbf{$\mathcal{I}$-soft convergence in soft topological spaces}}

In this section following Kostyrko et al. \cite{Kos1} and Lahiri et al. \cite{Lah1}, we introduce the notion of $\mathcal{I}$-soft convergence of sequences of soft points in a soft topological space. 

\begin{definition}
Let $(G_{M},\widetilde{\tau})$ be a soft topological space. A sequence $\{\eta_{n}\}_{n\in \mathbb{N}}$ of soft points in $G_{M}$ is 
said to be $\mathcal{I}$-soft convergent to a soft point $x^{M^{*}}_{s}$ in $G_{M}$, if for any soft neighbourhood $U_{M^{'}}$ of $x^{M^{*}}_{s}$, $\{n\in \mathbb{N}:
 \eta_{n}\widetilde{\notin} U_{M^{'}}\}\in \mathcal{I}$. 
In this case, $x^{M^{*}}_{s}$ is called $\mathcal{I}$-soft limit of $\{\eta_{n}\}_{n\in \mathbb{N}}$. In this case, we write $\mathcal{I}-\displaystyle{\lim_{n\rightarrow \infty}}{\eta_{n}}= x^{M^{*}}_{s}$.
\end{definition}

\begin{remark}
If we take $\mathcal{I}= \mathcal{I}_{f}=\{A\subset \mathbb N : A~\mbox{is a finite subset of}~\mathbb N\}$, then $\mathcal{I}_{f}$-soft convergent coincides with usual soft convergence \cite{Bay1} \cite{Hs1}. If we take $\mathcal{I}= \mathcal{I}_{d}=\{A\subset \mathbb N : d(A)=0 \}$, then $\mathcal{I}_{d}$-soft convergence coincides with statistical soft convergence \cite{Bay2}.
\end{remark}

\begin{note}
It is clear that for an admissible ideal $\mathcal{I}$ any soft convergent sequence of soft points in a soft topological space is $\mathcal{I}$-soft convergent with same limit. But the converse is not true. 
\end{note}


\begin{note}
Note that $\mathcal{I}$-soft limit of a sequence of soft points in a soft topological space may not be unique. 
\end{note}

If the soft topological space is soft Hausdorff then $\mathcal{I}$-soft limit of an $\mathcal{I}$-soft convergent sequence is unique, as proved in the following theorem.

\begin{theorem}\label{Thh2}
$\mathcal{I}$-soft limit of an $\mathcal{I}$-soft convergent sequence in a soft Hausdorff space is unique.
\end{theorem}

\begin{proof}
Let $(G_{M}, \widetilde{\tau})$ be a soft Hausdorff space and $\{\eta_{n}\}_{n\in \mathbb N}$ be a sequence of soft points in $G_{M}$, which is $\mathcal{I}$-soft convergent. If possible let  $\mathcal{I}-\displaystyle{\lim_{n\rightarrow \infty}}= x_{s}^{M^{*}}$ as well as $\mathcal{I}-\displaystyle{\lim_{n\rightarrow \infty}}= y_{s^{'}}^{M^{**}}$ in $G_{M}$. Then, there exists soft open sets $U_{M^{'}}$ and $V_{M^{''}}$ such that $x_{s}^{M^{*}}\widetilde{\in} U_{M^{'}},
 y_{s^{'}}^{M^{**}}\widetilde{\in} V_{M^{''}}$ and $U_{M^{'}}\widetilde{\cap} V_{M^{''}}= \widetilde{\emptyset}$. Then, $\{n\in \mathbb N: \eta_{n}\widetilde{\notin} U_{M^{'}}\}
\in \mathcal{I}$ and $\{n\in \mathbb N: \eta_{n}\widetilde{\notin} V_{M^{''}}\}\in \mathcal{I}$. So, $\{n\in \mathbb N: \eta_{n}\widetilde{\in} U_{M^{'}}\}\in \mathcal{F}(\mathcal{I})$ and $\{n\in \mathbb N: \eta_{n}\widetilde{\in} V_{M^{''}}\}\in \mathcal{F}(\mathcal{I})$. Then $\{n\in \mathbb N: \eta_{n}\widetilde{\in} U_{M^{'}}\}\cap \{n\in \mathbb N: \eta_{n}\widetilde{\in} V_{M^{''}}\} \in \mathcal{F}(\mathcal{I})$. This implies, $U_{M^{'}}\widetilde{\cap} V_{M^{''}}
 \neq \widetilde{\emptyset}$, which is a contradiction. Thus, $\{\eta_{n}\}_{n\in \mathbb N}$ has unique $\mathcal{I}$-soft limit.
\end{proof}

\begin{note}
Every subsequence of an $\mathcal{I}$-soft convergent sequence need not to be $\mathcal{I} $-soft convergent to the same $\mathcal{I}$-soft limit. 
\end{note}

\begin{theorem}
Let $(G_{M}, \widetilde{\tau})$ be a soft topological space and $\{\eta_{n}\}_{n\in \mathbb N}$ be a sequence of
soft points in $G_{M}$. If every subsequence of $\{\eta_{n}\}_{n\in \mathbb N}$ has a subsequence which is 
$\mathcal{I}$-soft convergent to a soft point $x_{s}^{M^{*}}\widetilde{\in} G_{M}$, then $\{\eta_{n}\}_{n\in \mathbb N}$
is $\mathcal{I}$-soft convergent to  $x_{s}^{M^{*}}$.
\end{theorem}

\begin{proof}
Proof is trivial, so omitted.
\end{proof}


\section{\textbf{$\mathcal{I}^{*}$-soft convergence in soft topological spaces}}
In this section following Kostyrko et al. \cite{Kos1} and Lahiri et al. \cite{Lah1}, we introduce the notion of $\mathcal{I}^{*}$-soft convergence of sequences of soft points in a soft topological space. 

\begin{definition}
Let $(G_{M},\widetilde{\tau})$ be a soft topological space. Then, a sequence $\{\eta_{n}\}_{n\in \mathbb N}$ of soft points in $G_{M}$ is said to be $\mathcal{I}^{*}$-soft convergent to a soft point $x_{s}^{M^{*}}$ in $G_{M}$ if there exists $P= \{p_{1}< 
p_{2}< \cdots< p_{k}<\cdots\}\in \mathcal{F}(\mathcal{I})$ such that $\displaystyle{\lim_{k\rightarrow \infty}}
 \eta_{p_{k}}= x_{s}^{M^{*}}$. In this case, we write $\mathcal{I}^{*}-\displaystyle{\lim_{n\rightarrow \infty}}= x_{s}^{M^{*}}$ and $x_{s}^{M^{*}}$ is called $\mathcal{I}^{*}$-soft limit of $\{\eta_{n}\}_{n\in \mathbb N}$.
\end{definition}
 
\begin{theorem}\label{Thh1}
Let $(G_{M},\widetilde{\tau})$ be a soft topological space and $\{\eta_{n}\}_{n\in \mathbb N}$ be a sequence of soft points in $G_{M}$. If $\mathcal{I}^{*}-\displaystyle{\lim_{n\rightarrow \infty}} {\eta_{n}}= x_{s}^{M^{*}}$, then $\mathcal{I}-\displaystyle{\lim_{n\rightarrow \infty}} {\eta_{n}}= x_{s}^{M^{*}}$ and if $(G_{M}, \widetilde{\tau})$ is soft Hausdorff then $\mathcal{I}^{*}$-soft limit of $\{\eta_{n}\}_{n\in \mathbb N}$ is unique.
\end{theorem}

\begin{proof}
Since $\mathcal{I}^{*}-\displaystyle{\lim_{n\rightarrow \infty}} {\eta_{n}}= x_{s}^{M^{*}}$, so there exists $P= \{p_{1}< p_{2}< \cdots< p_{k}< \cdots\}\in \mathcal{F}(\mathcal{I})$ such that $\displaystyle{\lim_{k\rightarrow \infty}} {\eta_{p_{k}}}= x_{s}^{M^{*}}$. Let $U_{M^{'}}$ be a soft neighbourhood of $x_{s}^{M^{*}}$. Then, there exists $k_{0}\in \mathbb N$, such that $\eta_{p_{k}}\widetilde{\in} U_{M^{'}},~\forall~k>k_{0}$. Now, $\{n\in \mathbb N: \eta_{n}\widetilde{\notin} U_{M^{'}}\}\subset (\mathbb N- P) \cup \{p_{1}, p_{2}, \cdots, p_{k_{0}}\}$. Since, $(\mathbb N- P)\cup \{p_{1}, p_{2}, \cdots, p_{k_{0}}\}\in \mathcal{I}$ so, $\{n\in \mathbb N: \eta_{n}\widetilde{\notin} U_{M^{'}}\}\in \mathcal{I}$. Therefore, $\mathcal{I}-\displaystyle{\lim_{n\rightarrow \infty}} {\eta_{n}}= x_{s}^{M^{*}}$. Now if $(G_{M}, \widetilde{\tau})$ is soft Hausdorff, then `Theorem \ref{Thh2}', $\mathcal{I}^{*}$-soft limit of $\{\eta_{n}\}_{n\in \mathbb N}$ is unique.

\end{proof}

\begin{theorem}
Let $(G_{M},\widetilde{\tau})$ be a soft topological space, having no soft limit point. Then $\mathcal{I}$ and $\mathcal{I}^{*}$ soft convergence coincide.
\end{theorem}

\begin{proof}
In the view of 'Theorem \ref{Thh1}', we have only to prove $\mathcal{I}$-soft convergence implies $\mathcal{I^{*}}$-soft convergence. 

Let $\{\eta_{n}\}_{n\in \mathbb N}$ be a sequence of soft points in $G_{M}$ such that  $\mathcal{I}-\displaystyle{\lim_{n\rightarrow \infty}} {\eta_{n}}= x_{s}^{M^{*}}$. Since, $(G_{M}, \widetilde{\tau})$ has no soft limit point, so $U_{M^{'}}=x_{s}^{M^{*}}$ is soft open. So, $\{n\in \mathbb N: \eta_{n}\widetilde{\notin} U_{M^{'}}\}\in \mathcal{I}$. Then, $\{n\in \mathbb N: \eta_{n}\widetilde{\in} U_{M^{'}}\}\in \mathcal{F}(\mathcal{I})$, i.e, $\{n\in \mathbb N: \eta_{n}=x_{s}^{M^{*}}\}\in \mathcal{F}(\mathcal{I})$. let $\{n\in \mathbb N: \eta_{n}=x_{s}^{M^{*}}\}= P= \{p_{1}< p_{2}< \cdots< p_{k}< \cdots \}$. Then $P\in \mathcal{F}(\mathcal{I})$ and $\displaystyle{\lim_{k\rightarrow \infty}} {\eta_{p_{k}}}= x_{s}^{M^{*}}$. So, $\mathcal{I}^{*}-\displaystyle{\lim_{n\rightarrow \infty}} {\eta_{n}}= x_{s}^{M^{*}}$.

\end{proof}

\begin{remark}
In an arbitrary soft topological space, equivalency between $\mathcal{I}$ and $\mathcal{I}^{*}$-soft convergence does not hold. 
\end{remark}

We now consider a condition namely AP condition under which $\mathcal{I}$ and $\mathcal{I^{*}}$ soft convergence coincide.This condition (AP) was introduced by Kostyrko et al. \cite{Kos1} which is similar to the (APO) condition used in \cite{Con} and \cite{Fast1}.

\begin{definition} \cite{Kos1}
An admissible ideal $\mathcal{I}$ of $\mathbb N$ is said to satisfy the condition AP, if for every 
countable family of mutually disjoint sets $\{H_{1}, H_{2},\cdots~\cdots\}$ belonging to
 $\mathcal{I}$, there exists a countable family of sets $\{K_{1}, K_{2},\cdots~\cdots\}$ such that $H_{i} \Delta K_{i}$ is finite, for all $i\in \mathbb N$ and $K=\overset{\infty}{\underset{i=1}{\cup}} K_{i}\in \mathcal{I}$. 
\end{definition}

\begin{theorem}
$(i)$ If $(G_{M},\widetilde{\tau})$ is a soft first countable soft topological space and $\mathcal{I}$ satisfies
 condition (AP), then for any sequence $\{\eta_{n}\}_{n\in \mathbb N}$ of soft
 points in $G_{M}$, $\mathcal{I}-\displaystyle{\lim_{n\rightarrow \infty}} {\eta_{n}}= x_{s}^{M^{*}}~\implies~\mathcal{I}^{*}-\displaystyle{\lim_{n\rightarrow \infty}} {\eta_{n}}= x_{s}^{M^{*}}$.

$(ii)$ If $(G_{M},\widetilde{\tau})$ is a soft first countable, soft $T_{1}$ space having at least one soft limit point and if for every sequence $\{\eta_{n}\}_{n\in \mathbb N}$ of soft points in $G_{M}$, $\mathcal{I}$-convergence of $\{\eta_{n}\}$ implies $\mathcal{I}^{*}$-convergence of $\{\eta_{n}\}$, then $\mathcal{I}$ has the property (AP).
\end{theorem}

\begin{proof}
$(i)$ Let $\{\eta_{n}\}_{n\in \mathbb N}$ be a sequence of soft points in $G_{M}$ such that $\mathcal{I}-\displaystyle{\lim_{n\rightarrow \infty}} {\eta_{n}}= x_{s}^{M^{*}}$, where $x_{s}^{M^{*}}\widetilde{\in} G_{M}$. Since $(G_{M}, \widetilde{\tau})$ is soft first countable so by `Theorem-\ref{Th1}', there exists a countable collection of soft local base $\{G_{n}(x_{s}^{M^{*}}): n\in \mathbb N\}$ at $x_{s}^{M^{*}}$, such that $G_{n+1}(x_{s}^{M^{*}})\widetilde{\subset} G_{n}(x_{s}^{M^{*}}),~\forall~n\in \mathbb N$.

Let $H_{1}=\{n\in \mathbb N: \eta_{n}\widetilde{\notin} G_{1}(x_{s}^{M^{*}})\}$ and $H_{m}=\{n\in \mathbb N: \eta_{n}\widetilde{\notin} G_{m}(x_{s}^{M^{*}})\widetilde{-}\\ G_{m-1}(x_{s}^{M^{*}})\}$, $\forall m\geq 2$. Then $H_{m}\in \mathcal{I},~\forall m \in \mathbb N$ and $H_{i}\cap H_{j}= \emptyset, \forall~i,j\in \mathbb N$. Since $\mathcal{I}$ satisfies (AP) condition, so there exists
  a countable collection of sets, say $\{K_{m}: m\in \mathbb N\}$ such that $H_{m}\Delta K_{m}$ is finite, $\forall m\in 
\mathbb N$ and $K=\underset{m\in \mathbb N}{\cup}{K_{m}}\in \mathcal{I}$. Then $R=\mathbb N - K=\{r_{1}< r_{2}<\cdots< r_{p}< \cdots~\cdots\} \in \mathcal{F(\mathcal{I})}$. Let $V_{M^{'}}$ be a soft neighbourhood of $x_{s}^{M^{*}}$. Since $\{\eta_{n}\}_{n\in \mathbb N}$ is $\mathcal{I}$-convergent to $x_{s}^{M^{*}}$ so $\{n\in \mathbb N:
 \eta_{n}\widetilde{\notin} V_{M^{'}}\}\in \mathcal{I}$. Then there exists $q\in \mathbb N$ such that $G_{n}(x_{s}^{M^{*}})
\widetilde{\subset}V_{M^{'}}, \forall~n \geq q$. Then, $\{n\in \mathbb N: \eta_{n}\widetilde{\notin} V_{M^{'}}\}\subset \underset{i=1}{\overset{q}{\cup}} H_{i} $. Since, $H_{m}\Delta K_{m}$ is finite, $\forall m\in \mathbb N$, so $\underset{i=1}{\overset{q}{\cup}} (H_{i} \Delta K_{i})$ is also finite. So, there exists $n_{0}\in \mathbb N$ such that $\forall~n\in \mathbb N$ with $n>n_{0}$, $n\in \underset{i=1}{\overset{q}{\cup}} H_{i}$ if and only if $n\in \underset{i=1}{\overset{q}{\cup}} K_{i}$. We Choose, $t\in \mathbb N$ such that $r_{t}> n_{0}$. Since $r_{j}\notin K\supset \underset{i=1}{\overset{q}{\cup}} K_{i}$ for any $j\in \mathbb N$, so if $p>t$, then $r_{p}> r_{t}>n_{0}$ and so $r_{p}\notin \{n\in \mathbb N : \eta_{n}\notin V_{M^{'}}\}^{C}$, $\forall p>t$ $\implies$ $r_{p}\in \{n\in \mathbb N : \eta_{n}\widetilde{\in} V_{M^{'}}\},~ \forall p>t$. So, $\displaystyle{\lim_{p\rightarrow \infty}}{\eta_{r_{p}}}= x_{s}^{M^{*}}$. Hence, $\mathcal{I}^{*}-\displaystyle{\lim_{n\rightarrow \infty}} {\eta_{n}}= x_{s}^{M^{*}}$.

$(ii)$ Let $x_{s}^{M^{*}}$ be a soft limit point of $G_{M}$. Since, $(G_{M}, \widetilde{\tau})$ is soft first countable, soft $T_{1}$ space, so there exist a soft local base $\{G_{n}(x_{s}^{M^{*}})\}_{n\in \mathbb N}$ at $x_{s}^{M^{*}}$ with $G_{n+1}(x_{s}^{M^{*}})\widetilde{\subset} G_{n}(x_{s}^{M^{*}}),~\forall~n\in \mathbb N$ and a sequence $\{\eta_{n}\}_{n\in \mathbb N}$ of distinct soft points in $G_{M}$ such that $\displaystyle{\lim_{n\rightarrow \infty}} \eta_{n}= x_{s}^{M^{*}}$, $\eta_{n}\widetilde{\in} G_{n}(x_{s}^{M^{*}})$, $\forall n\in \mathbb N$ and $\eta_{n} \neq x_{s}^{M^{*}}$ for any $n\in \mathbb N$. Let $\{H_{n}\}_{n\in \mathbb N}$ be a sequence of mutually disjoint non-empty sets from $\mathcal{I}$. Let us define a sequence $\{\gamma_{n}\}_{n\in \mathbb N}$ of soft points in $G_{M}$ by 
\begin{align*}
\gamma_{n}=\begin{cases}
\eta_{j}, &\mbox{if}~ n\in H_{j}\\
x_{s}^{M^{*}}, &\mbox{if}~ n\in \mathbb N-(\underset{j\in \mathbb N}{\cup}{H_{j}}).
\end{cases}
\end{align*}
Let $V_{M^{'}}$ be a soft neighbourhood of $x_{s}^{M^{*}}$. Then, there exists $q\in \mathbb N$, such that $G_{n}(x_{s}^{M^{*}})\widetilde{\subset} V_{M^{'}},~\forall~n\geq q$. So, $\{n\in \mathbb N: \gamma_{n}  \widetilde{\notin} V_{M^{'}}\}\subset H_{1}\cup H_{2}\cup \cdots \cup H_{q-1}$. Since, $H_{1}, H_{2},..., H_{q-1}\in \mathcal{I}$, so $H_{1}\cup H_{2}\cup \cdots \cup H_{q-1}\in \mathcal{I}$ and so $\{n\in \mathbb N: \gamma_{n}  \widetilde{\notin} V_{M^{'}}\}\in \mathcal{I}$. So, $\mathcal{I}$-$\displaystyle{\lim_{n\rightarrow \infty}} {\gamma_{n}}= x_{s}^{M^{*}}$. Then, by the given condition, $\mathcal{I}^{*}$-$\displaystyle{\lim_{n\rightarrow \infty}} {\gamma_{n}}= x_{s}^{M^{*}}$. Then, there exists  $R= \{r_{1}< r_{2}< \cdots< r_{k}< \cdots\}\in \mathcal{F}(\mathcal{I})$ such that $\displaystyle{\lim_{k\rightarrow \infty}} {\gamma_{r_{k}}}= x_{s}^{M^{*}}$. Let $H=\mathbb N - R$. Then $H\in \mathcal{I}$. Let $K_{j}= H_{j}\cap H,~\forall j\in \mathbb N$. Since, $H\in \mathcal{I}$ and $K_{j}\subset H,~\forall j\in \mathbb N$, so $K_{j}\in \mathcal{I},~\forall j\in \mathbb N$ and since $\underset{j\in \mathbb N}{\cup} {K_{j}}\subset H$, $\underset{j\in \mathbb N}{\cup} {K_{j}}\in \mathcal{I}$.

 Let $j\in \mathbb N$. If $H_{j}\cap R$ is infinite, then $\gamma_{r_{k}}= \eta_{j}$ for infinitely many $r_{k}$'s. But $\eta_{j}\neq x_{s}^{M^{*}}$ and so $G_{M}\widetilde{-} \{\eta_{j}\}$ is a soft neighbourhood of $x_{s}^{M^{*}}$ and $\gamma_{r_{k}}\widetilde{\notin} G_{M}\widetilde{-}\{\eta_{j}\}$ for infinitely many $r_{k}$'s which contradicts that $\displaystyle{\lim_{k\rightarrow \infty}} {\gamma_{r_{k}}}= x_{s}^{M^{*}}$. So $H_{j}\cap R$ is a finite set.  

So, there exists $l_{0}\in \mathbb N$ such that $H_{j}\subset (H_{j}\cap K_{j})\cup \{r_{1}, r_{2}, \cdots, r_{l_{0}}\}$. Then, $H_{j}\Delta K_{j}= H_{j}- K_{j}\subset \{r_{1}, r_{2}, \cdots, r_{l_{0}}\}$. So, $H_{j}\Delta K_{j}$ is finite.  

Thus, we get a countable family of sets $\{K_{j}: j\in \mathbb N\}$ such that $H_{j}\Delta K_{j}$ is finite, $\forall j\in \mathbb N$ and $\underset{j\in \mathbb N}{\cup}{K_{j}}\in  \mathcal{I}$. So, `$\mathcal{I}$' satisfies (AP) condition.
\end{proof}

\section{\textbf{$\mathcal{I}$-soft limit point and $\mathcal{I}$-soft cluster point }}

In this section, following Kostyrko et al. \cite{Kos1} and Lahiri et al. \cite{Lah1}, we introduce the notion of $\mathcal{I}$-soft limit point and $\mathcal{I}$-soft cluster point of a sequence of soft points in a soft topological space.

\begin{definition}
Let $(G_{M},\widetilde{\tau})$ be a soft topological space and $w=\{\eta_{n}\}_{n\in \mathbb N}$ be a sequence of soft points in $G_{M}$. A soft
 point $x_{s}^{M^{*}}$ is said to be an $\mathcal{I}$-soft limit point of $w=\{\eta_{n}\}_{n\in \mathbb N}$, if there exists a subset $P=\{p_{1}< p_{2}< \cdots <p_{k}< \cdots\}$ of
 $\mathbb N$ such that $P\notin \mathcal{I}$ and $\displaystyle{\lim_{k\rightarrow \infty}} \eta_{p_{k}}= x_{s}^{M^{*}}$. 

The soft set of all $\mathcal{I}$-soft limit points of $w=\{\eta_{n}\}_{n\in \mathbb N}$ is denoted by ${\mathcal{I}}({\widetilde{\Lambda}}_{w})$.
\end{definition}

\begin{definition}
Let $(G_{M},\widetilde{\tau})$ be a soft topological space and $w=\{\eta_{n}\}_{n\in \mathbb N}$ be a sequence of soft points in $G_{M}$. A soft
 point $x_{s}^{M^{*}}$ is said to be an $\mathcal{I}$-soft cluster point of $w=\{\eta_{n}\}_{n\in \mathbb N}$, if for any soft neighbourhood $U_{M^{''}}$ of $x_{s}^{M^{*}}$, $\{n\in \mathbb N
: \eta_{n}\widetilde{\in} U_{M^{''}}\}\notin \mathcal{I}$.

The soft set of all $\mathcal{I}$-soft cluster points of $w=\{\eta_{n}\}_{n\in \mathbb N}$ is denoted by ${\mathcal{I}}({\widetilde{\Gamma}}_{w})$. 
\end{definition}

\begin{theorem}
Let $(G_{M},\widetilde{\tau})$ be a soft topological space. Then, for any sequence $w=\{\eta_{n}\}_{n\in \mathbb N}$ of soft points in $G_{M}$, we have ${\mathcal{I}}({\widetilde{\Lambda}}_{w})\widetilde{\subset} {\mathcal{I}}({\widetilde{\Gamma}}_{w})$.
\end{theorem}

\begin{proof}
Let $x_{s}^{M^{*}}\widetilde{\in} {\mathcal{I}}({\widetilde{\Lambda}}_{w})$. Then, there exists a set $P= \{p_{1}< p_{2}< \cdots< p_{k}< \cdots\}\notin \mathcal{I}$ such
 that $\displaystyle{\lim_{k\rightarrow \infty}} {\eta_{p_{k}}} = x_{s}^{M^{*}}$. Let $U_{M^{'}}$ be a soft neighbourhood of $x_{s}^{M^{*}}$. Then, there exists $k_{0}\in \mathbb N$ such that $\eta_{p_{k}}\widetilde{\in} U_{M^{'}},~\forall~k>k_{0}$. Then, $\{n\in \mathbb N: \eta_{n}\widetilde{\in} U_{M^{'}}\}\supset P-\{p_{1}, p_{2},\cdots, p_{k_{0}}\}$. If $\{n\in \mathbb N: \eta_{n}\widetilde{\in} U_{M^{'}}\}\in \mathcal{I}$, then, $P-\{p_{1}, p_{2},\cdots, p_{k_{0}}\}\in \mathcal{I}$ and so $P= (P-\{p_{1}, p_{2},\cdots, p_{k_{0}}\})\cup \{p_{1}, p_{2},\cdots, p_{k_{0}}\}\in \mathcal{I}$, a contradiction. Therefore, $\{n\in \mathbb N: \eta_{n}\widetilde{\in} U_{M^{'}}\}\notin \mathcal{I}$. So, $x_{s}^{M^{*}}$ is an $\mathcal{I}$-soft cluster point of $\{\eta_{n}\}_{n\in \mathbb N}$. Hence ${\mathcal{I}}({\widetilde{\Lambda}}_{w})\widetilde{\subset} {\mathcal{I}}({\widetilde{\Gamma}}_{w})$.
\end{proof}




\begin{theorem}
Let $(G_{M},\widetilde{\tau})$ be a soft topological space. For any sequence $w=\{\eta_{n}\}_{n\in \mathbb N}$ of soft points in $G_{M}$, ${\mathcal{I}}({\widetilde{\Gamma}}_{w})$ is soft closed in $(G_{M}, \widetilde{\tau})$.
\end{theorem}

\begin{proof}
Let $y_{s}^{M^{**}}\widetilde{\in} 
~\mbox{cl}({\mathcal{I}}({\widetilde{\Gamma}}_{w}))$ and let $U_{M^{''}}$ be a soft neighbourhood of $y_{s}^{M^{**}}$. Then, $U_{M^{''}}\widetilde{\cap}
 {\mathcal{I}}({\widetilde{\Gamma}}_{w}) \neq \widetilde{\emptyset}$. Let $w_{s^{'}}^{M^{***}}\widetilde{\in}
 {\mathcal{I}}({\widetilde{\Gamma}}_{w})\widetilde{\cap} U_{M^{''}}$. Since, $w_{s^{'}}^{M^{***}}$ is an $\mathcal{I}$-soft cluster point of $\{\eta_{n}\}_{n\in \mathbb N}$ and $U_{M^{''}}$ is a soft neighbourhood of $w_{s^{'}}^{M^{***}}$, so $\{n\in \mathbb N: \eta_{n}\widetilde{\in} U_{M^{''}}\}
\notin \mathcal{I}$. So, $y_{s}^{M^{**}}\widetilde{\in} {\mathcal{I}}({\widetilde{\Gamma}}_{w})$.  ${\mathcal{I}}({\widetilde{\Gamma}}_{w})$ is soft closed in $(G_{M}, \widetilde{\tau})$.  
\end{proof}

\begin{theorem}
Let $(G_{M}, \widetilde{\tau})$ be a soft topological space and $w=\{\eta_{n}\}_{n\in \mathbb N}$ and $v=\{\gamma_{n}\}_{n\in \mathbb N}$ are sequences of soft points in $G_{M}$ such that $\{n\in \mathbb N : \eta_{n}\neq \gamma_{n}\}\in \mathcal{I}$. Then 

$(i)$ ${\mathcal{I}}({\widetilde{\Lambda}}_{w})= {\mathcal{I}}({\widetilde{\Lambda}}_{v})$

$(ii)$ ${\mathcal{I}}({\widetilde{\Gamma}}_{w})= {\mathcal{I}}({\widetilde{\Gamma}}_{v})$.

\end{theorem}

\begin{proof}
$(i)$ Let $x_{s}^{M^{*}}\widetilde{\in} {\mathcal{I}}({\widetilde{\Lambda}}_{w})$. Then there exists a set $W=\{w_{1}< w_{2}< \cdots< w_{k}< \cdots\}\subset \mathbb N$ such that $W\notin \mathcal{I}$ and $\displaystyle{\lim_{k\rightarrow \infty}} \eta_{w_{k}}= x_{s}^{M^{*}}$.

Let $Y=\{n\in \mathbb N : \eta_{n}\neq \gamma_{n}\}$. We claim that $W-Y\notin \mathcal{I}$. In contrary, suppose $W-Y\in \mathcal{I}$. Since $Y\in \mathcal{I}$, so $W= (W-Y)\cup Y\in \mathcal{I}$, a contradiction. So, $W-Y\notin \mathcal{I}$. Then $W-Y$ is an infinite set. Let $W-Y=\{t_{1}< t_{2}< \cdots< t_{l}< \cdots\}$ and $V_{M^{'}}$ be a soft neighbourhood of $x_{s}^{M^{*}}$. Since $\displaystyle{\lim_{k\rightarrow \infty}}{\eta_{w_{k}}}= x_{s}^{M^{*}}$, so there exists $k_{0}\in \mathbb N$, such that $\eta_{w_{k}}\widetilde{\in} V_{M^{'}},~\forall~k> k_{0}$. Then, their exists $k_{0}^{'}\in \mathbb N$ such that $\gamma_{t_{l}}\widetilde{\in} V_{M^{'}},~\forall~l> k_{0}^{'}$. Therefore, $\displaystyle{\lim_{l\rightarrow \infty}} \gamma_{t_{l}}= x_{s}^{M^{*}}$ and hence $x_{s}^{M^{*}}\widetilde{\in} {\mathcal{I}}({\widetilde{\Lambda}}_{v})$. So, ${\mathcal{I}}({\widetilde{\Lambda}}_{w})\widetilde{\subset} {\mathcal{I}}({\widetilde{\Lambda}}_{v})$. Similarly, we have ${\mathcal{I}}({\widetilde{\Lambda}}_{v})\widetilde{\subset} {\mathcal{I}}({\widetilde{\Lambda}}_{w})$. Thus, ${\mathcal{I}}({\widetilde{\Lambda}}_{w})= {\mathcal{I}}({\widetilde{\Lambda}}_{v})$.

$\\(ii)$ Let $x_{s}^{M^{*}}\widetilde{\in} {\mathcal{I}}({\widetilde{\Gamma}}_{w})$ and $R_{M^{''}}$ be a soft neighbourhood of $x_{s}^{M^{*}}$. Then $\{n\in \mathbb N : \eta_{n}\widetilde{\in} R_{M^{''}}\}\notin \mathcal{I}$. If possible, let $\{n\in \mathbb N : \gamma_{n}\widetilde{\in} R_{M^{''}}\}\in \mathcal{I}$.

Since $\{n\in \mathbb N : \eta_{n}\widetilde{\in} R_{M^{''}}\}\subset \{n\in \mathbb N : \eta_{n}\neq \gamma_{n}\}\cup\{n\in \mathbb N : \gamma_{n}\widetilde{\in} R_{M^{''}}\}$ and $\{n\in \mathbb N : \eta_{n}\neq \gamma_{n}\}\in \mathcal{I}$, so $\{n\in \mathbb N : \eta_{n}\neq \gamma_{n}\}\cup\{n\in \mathbb N : \gamma_{n}\widetilde{\in} R_{M^{''}}\}\in \mathcal{I}$ and so $\{n\in \mathbb N : \eta_{n}\widetilde{\in} R_{M^{''}}\}\in \mathcal{I}$, a contradiction. Therefore, $\{n\in \mathbb N : \gamma_{n}\widetilde{\in} R_{M^{''}}\}\notin \mathcal{I}$. This implies, $x_{s}^{M^{*}}\widetilde{\in} {\mathcal{I}}({\widetilde{\Gamma}}_{v})$. So, ${\mathcal{I}}({\widetilde{\Gamma}}_{w})\widetilde{\subset} {\mathcal{I}}({\widetilde{\Gamma}}_{v})$. Similarly we have ${\mathcal{I}}({\widetilde{\Gamma}}_{v})\widetilde{\subset} {\mathcal{I}}({\widetilde{\Gamma}}_{w})$. Hence, ${\mathcal{I}}({\widetilde{\Gamma}}_{w})= {\mathcal{I}}({\widetilde{\Gamma}}_{v})$.

\end{proof}




\bigskip
\noindent
{\bf Acknowledgment.}
The second author is grateful to the University Grants Commission, India for
financial support under UGC-JRF scheme during the preparation of this paper.

\end{document}